\documentclass[secthm,seceqn,amsthm,ussrhead,12pt]{amsart}
\usepackage[utf8]{inputenc}
\usepackage[english]{babel}
\usepackage{amssymb,amsmath,amsthm,amsfonts,xcolor,enumerate,hyperref,comment,longtable,cleveref}

\usepackage{times}
\usepackage{cite}
\usepackage{pdflscape}
\usepackage{ulem}
\usepackage[mathcal]{euscript}
\usepackage{tikz}
\usepackage{hyperref}
\usepackage{cancel}
\usepackage{stmaryrd}

\usetikzlibrary{arrows}

\newtheorem{theorem}{Theorem}
\newtheorem{lemma}[theorem]{Lemma}
\newtheorem{proposition}[theorem]{Proposition}

\newtheorem{definition}[theorem]{Definition}

\newenvironment{Proof}[1][Proof.]{\begin{trivlist}
		\item[\hskip \labelsep {\bfseries #1}]}{\flushright
		$\Box$\end{trivlist}}

\usepackage{stmaryrd}
\usepackage{xcolor}

\newcommand{\aut}[1]{\operatorname{\mathrm{Aut}}{#1}}

\newcommand{\ann}[1]{\operatorname{\mathrm{Ann}}{#1}}

\begin{document}
	
	\noindent{\large
		Central extensions of filiform Zinbiel algebras}
	\footnote{
	    This work was supported by Agencia Estatal de Investigación (Spain), grant MTM2016-79661-P (European FEDER support included, UE);   RFBR 20-01-00030; FAPESP  18/15712-0,  18/12197-7, 	19/00192-3.
	}

	\ 
	
	{\bf  
		Luisa M. Camacho$^{a}$,
		Iqboljon Karimjanov$^{b},$ 
		Ivan   Kaygorodov$^{c}$ \&
		Abror Khudoyberdiyev$^{d}$}

	\

	\

	{\tiny
		
		$^{a}$ University of Sevilla, Sevilla, Spain.
		
		$^{b}$ Andijan State University, Andijan, Uzbekistan.
		
		$^{c}$ CMCC, Universidade Federal do ABC, Santo Andr\'e, Brazil.
		
		$^{d}$ National University of Uzbekistan, Institute of Mathematics Academy of
		Sciences of Uzbekistan, Tashkent, Uzbekistan.
		
	\ 
	
		\smallskip
		
		E-mail addresses:

		\smallskip
		
		Luisa M. Camacho (lcamacho@us.es)
		
		Iqboljon Karimjanov (iqboli@gmail.com)
		
		Ivan   Kaygorodov (kaygorodov.ivan@gmail.com)
		
		Abror Khudoyberdiyev (khabror@mail.ru)

	}

	\ 
	
	\
	
	\noindent{\bf Abstract}:
	{\it  In this paper we describe central extensions (up to isomorphism) of all complex null-filiform and filiform Zinbiel algebras. It is proven that every non-split central extension of an $n$-dimensional null-filiform Zinbiel algebra is isomorphic to an $(n+1)$-dimensional null-filiform Zinbiel algebra. Moreover, we obtain all pairwise non isomorphic quasi-filiform Zinbiel algebras.}
	
	\

	
	\ 
	

	\ 
	
	\ 
	
	\section*{Introduction}

	The algebraic classification (up to isomorphism) of an $n$-dimensional algebras from a certain variety
	defined by some family of polynomial identities is a classical problem in the theory of non-associative algebras.
	There are many results related to algebraic classification of small dimensional algebras in the varieties of 
	Jordan, Lie, Leibniz, Zinbiel and many another algebras \cite{ack,jkk19,kkl20,gkks, degr3, lisa, cfk19, usefi1, degr1, degr2,kl20, kv16, ckkk19,kks19, ha16, hac18}.
	An algebra $\bf A$ is called a {\it Zinbiel algebra} if it satisfies the identity 
	\[(x\circ y)\circ z=x\circ (y\circ z+z\circ y).\]
	Zinbiel algebras were introduced by Loday in \cite{loday} and studied in \cite{abash,adashev19, cam13,mukh, dok,  anau, dzhuma, dzhuma5, dzhuma19, kppv, ualbay, yau}. Under the Koszul duality, the operad of Zinbiel algebras is dual to the operad of Leibniz algebras. 
	Hence,  the tensor product of a Leibniz algebra and a Zinbiel algebra can be given the structure of a Lie algebra. Under the symmetrized product, a Zinbiel algebra becomes an associative and commutative algebra. Zinbiel algebras are also related to Tortkara algebras \cite{dzhuma} and Tortkara triple systems \cite{brem}. 
	More precisely, every Zinbiel algebra with the commutator multiplication gives a Tortkara algebra (also about Tortkara algebras, see, \cite{gkp,gkk,gkks}),
	which have recently sprung up in unexpected areas of mathematics \cite{tortnew1,tortnew2}.
	
	Central extensions play an important role in quantum mechanics: one of the earlier
	encounters is by means of Wigner's theorem which states that a symmetry of a quantum
	mechanical system determines an (anti-)unitary transformation of a Hilbert space.
	Another area of physics where one encounters central extensions is the quantum theory
	of conserved currents of a Lagrangian. These currents span an algebra which is closely
	related to so-called affine Kac-Moody algebras, which are universal central extensions
	of loop algebras.
	Central extensions are needed in physics, because the symmetry group of a quantized
	system usually is a central extension of the classical symmetry group, and in the same way
	the corresponding symmetry Lie algebra of the quantum system is, in general, a central
	extension of the classical symmetry algebra. Kac-Moody algebras have been conjectured
	to be symmetry groups of a unified superstring theory. The centrally extended Lie
	algebras play a dominant role in quantum field theory, particularly in conformal field
	theory, string theory and in $M$-theory.
	In the theory of Lie groups, Lie algebras and their representations, a Lie algebra extension
	is an enlargement of a given Lie algebra $g$ by another Lie algebra $h.$ Extensions
	arise in several ways. There is a trivial extension obtained by taking a direct sum of
	two Lie algebras. Other types are a split extension and a central extension. Extensions may
	arise naturally, for instance, when forming a Lie algebra from projective group representations.
	A central extension and an extension by a derivation of a polynomial loop algebra
	over a finite-dimensional simple Lie algebra gives a Lie algebra which is isomorphic to a
	non-twisted affine Kac-Moody algebra \cite[Chapter 19]{bkk}. Using the centrally extended loop
	algebra one may construct a current algebra in two spacetime dimensions. The Virasoro
	algebra is the universal central extension of the Witt algebra, the Heisenberg algebra is
	the central extension of a commutative Lie algebra  \cite[Chapter 18]{bkk}.
	
	The algebraic study of central extensions of Lie and non-Lie algebras has a very long history \cite{omirov,ha17,ha16nm,hac16,kkl18,is11,ss78,zusmanovich,klp20}.
	For example, all central extensions of some filiform Leibniz algebras were classified in \cite{is11,omirov}
	and all central extensions of filiform associative algebras were classified in \cite{kkl18}.
	Skjelbred and Sund used central extensions of Lie algebras for a classification of low dimensional nilpotent Lie algebras  \cite{ss78}.
	After that, the method introduced by Skjelbred and Sund was used to 
	describe  all non-Lie central extensions of  all $4$-dimensional Malcev algebras \cite{hac16},
	all non-associative central extensions of $3$-dimensional Jordan algebras \cite{ha17},
	all anticommutative central extensions of $3$-dimensional anticommutative algebras \cite{cfk182}.
	Note that the method of central extensions is an important tool in the classification of nilpotent algebras.
	It was used to describe
	all $4$-dimensional nilpotent associative algebras \cite{degr1},
	all $4$-dimensional nilpotent assosymmetric  algebras \cite{ikm19},
	all $4$-dimensional nilpotent bicommutative algebras \cite{kpv19},
	all $4$-dimensional nilpotent Novikov  algebras \cite{kkk18},
	all $4$-dimensional commutative algebras \cite{fkkv19},
	all $5$-dimensional nilpotent Jordan algebras \cite{ha16},
	all $5$-dimensional nilpotent restricted Lie algebras \cite{usefi1},
	all $5$-dimensional anticommutative algebras \cite{fkkv19},
	all $6$-dimensional nilpotent Lie algebras \cite{degr3,degr2},
	all $6$-dimensional nilpotent Malcev algebras \cite{hac18},
	all $6$-dimensional nilpotent binary Lie algebras\cite{ack},  
	all $6$-dimensional nilpotent anticommutative $\mathfrak{CD}$-algebras \cite{ack},
	all $6$-dimensional nilpotent Tortkara algebras\cite{gkks, gkk},  
	and some others.

	\section{Preliminaries}
	
	All algebras and vector spaces in this paper are over $\mathbb{C}.$
	
	\subsection{Filiform Zinbiel algebras}
	
	An algebra $\bf {A}$ is called \textit{Zinbiel algebra} if for any $x,y,z\in \bf{A}$ it satisfies the identity 
	$$(x\circ y)\circ z=x\circ (y\circ z)+x\circ (z\circ y).$$

	For an algebra $\bf {A}$, we consider the series
	\[
	{\bf A}^1={\bf A}, \qquad \ {\bf A}^{i+1}=\sum\limits_{k=1}^{i}{\bf A}^k {\bf A}^{i+1-k}, \qquad i\geq 1.
	\]
	
	We say that  an  algebra $\bf A$ is {\it nilpotent} if ${\bf  A}^{i}=0$ for some $i \in \mathbb{N}$. The smallest integer satisfying ${\bf A}^{i}=0$ is called the  {\it nilpotency index} of $\bf A$.
	
	\begin{definition}
		An $n$-dimensional algebra $\bf A$ is called null-filiform if $\dim {\bf A}^i=(n+ 1)-i,\ 1\leq i\leq n+1$.
	\end{definition}
	
	It is easy to see that a Zinbiel algebra has a maximal nilpotency index if and only if it is null-filiform. 
	For a nilpotent Zinbiel  algebra, the condition of null-filiformity is equivalent to the condition that the algebra is one-generated.
	
	All null-filiform Zinbiel algebras were described in  \cite{adashev}. Throughout the paper, $C_i^j$ denotes the combinatorial numbers $\binom{i}{j}.$
	
	\begin{theorem} \cite{adashev} \label{nullfil}  An arbitrary $n$-dimensional null-filiform Zinbiel algebra is isomorphic to the algebra $F_n^0$:
		\[  e_i\circ e_j=C_{i+j-1}^{j}  e_{i+j}, \quad 2\leq i+j\leq n,\]
		where omitted products are equal to zero and $\{ e_1, e_2, \dots, e_n\}$ is a basis of the algebra.

	\end{theorem}
	
	As an easy corollary from the previous theorem we have the next result.
	
	\begin{theorem} 
		Every non-split central extension of $F_n^0$ is isomorphic to $F_{n+1}^0.$
	\end{theorem}
	
	\begin{Proof}
		It is easy to see, that every non-split central extension of $F_n^0$ is a one-generated nilpotent algebra.
		It follows that every non-split central extension of a null-filiform Zinbiel algebra is a null-filiform Zinbiel algebra.
		Using the classification of null-filiform algebras (Theorem \ref{nullfil}) we have the statement of the Theorem.
	\end{Proof}

	\begin{definition}
		An $n$-dimensional algebra is called filiform if $\dim (\mathbf{A} ^i)=n-i, \ 2\leq i \leq n$.
	\end{definition}

	All filiform Zinbiel algebras were classified in  \cite{adashev}.
	
\begin{theorem}\label{th-filiform}
		An arbitary $n$-dimensional ($n\geq 5$)  filiform Zinbiel algebra is isomorphic to one of the following pairwise non-isomorphic algebras:
\begin{longtable}{lllll}
		$F_n^1$ &$:$ & $e_i\circ e_j=C_{i+j-1}^{j} e_{i+j},$  & $2\leq i+j\leq n-1;$ &\\
		$F_n^2$ &$:$ & $e_i\circ e_j=C_{i+j-1}^{j} e_{i+j},$  &$2\leq i+j\leq n-1,$& $e_n\circ e_1=e_{n-1};$ \\
		$F_n^3$&$:$ & $e_i \circ e_j=C_{i+j-1}^{j} e_{i+j},$  &$2\leq i+j\leq n-1,$& $e_n\circ e_n=e_{n-1}.$ 
		\end{longtable}
	\end{theorem}

	\subsection{Basic definitions and methods}
	Throughout this paper, we are using the notations and methods well written in \cite{ha17,hac16}
	and adapted for the Zinbiel case with some modifications.
	From now, we will give only some important definitions.
	
	Let $({\bf A}, \circ)$ be a Zinbiel algebra and $\mathbb V$ a vector space. Then the $\mathbb C$-linear space ${\rm Z}^{2}\left(
	\bf A,\mathbb V \right) $ is defined as the set of all  bilinear maps $\theta  \colon {\bf A} \times {\bf A} \longrightarrow {\mathbb V}$,
	such that 
	\[\theta(x\circ y,z)=\theta(x,y\circ z+z\circ y).\]
	Its elements will be called {\it cocycles}. For a
	linear map $f$ from $\bf A$ to  $\mathbb V$, if we write $\delta f\colon {\bf A} \times
	{\bf A} \longrightarrow {\mathbb V}$ by $\delta f  (x, y) =f(x \circ y)$, then $\delta f\in {\rm Z}^{2}\left( {\bf A},{\mathbb V} \right) $. We define ${\rm B}^{2}\left(
	{\bf A},{\mathbb V}\right) =\left\{ \theta =\delta f\ : f\in 
	{\rm Hom}\left( {\bf A},{\mathbb V}\right) \right\} $.
	One can easily check that ${\rm B}^{2}(\bf A,\mathbb V)$ is a linear subspace of ${\rm Z}^{2}\left( {\bf A},{\mathbb V}\right) $ whose elements are called
	{\it coboundaries}. We define the {\it second cohomology space} ${\rm H}^{2}\left( {\bf A},{\mathbb V}\right) $ as the quotient space ${\rm Z}^{2}
	\left( {\bf A},{\mathbb V}\right) \big/{\rm B}^{2}\left( {\bf A},{\mathbb V}\right) $.
 
	\
	
	Let $\aut(\mathbf{A}) $ be the automorphism group of the Zinbiel  algebra $\mathbf{A} $ and let $\phi \in \aut(\mathbf{A})$. For $\theta \in
	{\rm Z}^{2}\left( {\bf A},{\mathbb V}\right) $ define $\phi \theta (x,y)
	=\theta \left( \phi \left( x\right) ,\phi \left( y\right) \right) $. Then $\phi \theta \in {\rm Z}^{2}\left( {\bf A},{\mathbb V}\right) $. So, $\aut(\mathbf{A})$
	acts on ${\rm Z}^{2}\left( {\bf A},{\mathbb V}\right) $. It is easy to verify that
	${\rm B}^{2}\left( {\bf A},{\mathbb V}\right) $ is invariant under the action of $\aut(\mathbf{A})$ and so we have that $\aut(\mathbf{A})$ acts on ${\rm H}^{2}\left( {\bf A},{\mathbb V}\right)$.
	
	\
	
	Let $\bf A$ be a Zinbiel  algebra of dimension $m<n,$ and ${\mathbb V}$ be a $\mathbb C$-vector
	space of dimension $n-m$. For any $\theta \in {\rm Z}^{2}\left(
	{\bf A},{\mathbb V}\right) $ define on the linear space ${\bf A}_{\theta } := {\bf A}\oplus {\mathbb V}$ the
	bilinear product `` $\left[ -,-\right] _{{\bf A}_{\theta }}$'' by $\left[ x+x^{\prime },y+y^{\prime }\right] _{{\bf A}_{\theta }}=
	x\circ y +\theta(x,y) $ for all $x,y\in {\bf A},x^{\prime },y^{\prime }\in {\mathbb V}$.
	The algebra ${\bf A}_{\theta }$ is a Zinbiel  algebra which is called an $(n-m)$-{\it dimensional central extension} of ${\bf A}$ by ${\mathbb V}$. Indeed, we have, in a straightforward way, that ${\bf A_{\theta}}$ is a Zinbiel algebra if and only if $\theta \in {\rm Z}^{2}({\bf A}, {\mathbb V})$.
	
	We also call the
	set $\ann(\theta)=\left\{ x\in {\bf A}:\theta \left( x, {\bf A} \right)+ \theta \left({\bf A} ,x\right) =0\right\} $
	the {\it annihilator} of $\theta $. We recall that the {\it annihilator} of an  algebra ${\bf A}$ is defined as
	the ideal $\ann(  \mathbf{A} ) =\left\{ x\in {\bf A}:  x \circ {\bf A}+ {\bf A}\circ x =0\right\}$ and observe
	that
	$\ann({\bf A}_{\theta}) = (\ann(\theta) \cap \ann(\mathbf{A}) )\oplus \mathbb{V}.$
	
	\
	
	We have the next  key result:
	
	\begin{lemma}
		Let ${\bf A}$ be an $n$-dimensional Zinbiel algebra such that $\dim(\ann({\bf A}))=m\neq0$. Then there exists, up to isomorphism, a unique $(n-m)$-dimensional Zinbiel  algebra $\mathbf{A}'$ and a bilinear map 
		$\theta \in {\rm Z}^{2}({\bf A}, {\mathbb V})$ with $\ann({\bf A})\cap \ann(\theta)=0$, where $\mathbb{V}$ is a vector space of dimension $m,$ such that $\mathbf{A} \cong {\mathbf{A}'}_{\theta}$ and
		${\bf A}/\ann({\bf A})\cong \mathbf{A}'$.
	\end{lemma}

	\
	
	However, in order to solve the isomorphism problem we need to study the
	action of $\aut(\mathbf{A})$ on ${\rm H}^{2}\left( {\bf A},{\mathbb V}
	\right) $. To do that, let us fix $e_{1},\ldots ,e_{s}$ a basis of ${\mathbb V}$, and $
	\theta \in {\rm Z}^{2}\left( {\bf A},{\mathbb V}\right) $. Then $\theta $ can be uniquely
	written as $\theta \left( x,y\right) =
	\displaystyle \sum_{i=1}^{s} \theta _{i}\left( x,y\right) e_{i}$, where $\theta _{i}\in
	{\rm Z}^{2}\left( {\bf A},\mathbb C \right) $. Moreover, $\ann(\theta)=\ann(\theta _{1})\cap \ann(\theta _{2})\cap\ldots \cap \ann(\theta _{s})$. Further, $\theta \in
	{\rm B}^{2}\left( {\bf A},{\mathbb V}\right) $\ if and only if all $\theta _{i}\in {\rm B}^{2}\left( {\bf A},
	\mathbb C\right) $.
	
	\;

	\begin{definition}
		Let ${\bf A}$ be an algebra and $I$ be a subspace of ${\rm Ann}({\bf A})$. If ${\bf A}={\bf A}_0 \oplus I$
		then $I$ is called an {\it annihilator component} of ${\bf A}$.
	\end{definition}
	
	\begin{definition}
		A central extension of an algebra $\bf A$ without annihilator component is called a non-split central extension.
	\end{definition}
	
	It is not difficult to prove (see \cite[Lemma 13]{hac16}), that given a Zinbiel  algebra ${\bf A}_{\theta}$, if we write as
	above $\theta \left( x,y \right) = \displaystyle \sum_{i=1}^{s} \theta_{i}\left( x,y \right) e_{i}\in {\rm Z}^{2}\left( {\bf A},{\mathbb V} \right) $ and we have
	$\ann(\theta)\cap \ann( {\bf A}) =0$, then 
	${\bf A}_{\theta}$ has an
	annihilator component if and only if $\left[ \theta _{1}\right] ,\left[
	\theta _{2}\right] ,\ldots ,\left[ \theta _{s}\right] $ are linearly
	dependent in ${\rm H}^{2}\left( {\bf A},\mathbb C\right) $.
	
	\;
	
	Let ${\mathbb V}$ be a finite-dimensional vector space. The {\it Grassmannian} $G_{k}\left( {\mathbb V}\right) $ is the set of all $k$-dimensional
	linear subspaces of $ {\mathbb V}$. Let $G_{s}\left( {\rm H}^{2}\left( {\bf A},{\mathbb C}\right) \right) $ be the Grassmannian of subspaces of dimension $s$ in
	${\rm H}^{2}\left( {\bf A},{\mathbb C}\right) $. There is a natural action of $\aut(\mathbf{A})$ on $G_{s}\left( {\rm H}^{2}\left( {\bf A},{\mathbb C}\right) \right) $.
	Let $\phi \in \aut(\mathbf{A})$. For $W=\left\langle
	\left[ \theta _{1}\right] ,\left[ \theta _{2}\right] ,\dots,\left[ \theta _{s}
	\right] \right\rangle \in G_{s}\left( {\rm H}^{2}\left( {\bf A},{\mathbb C}
	\right) \right) $ define $\phi W=\left\langle \left[ \phi \theta _{1}\right]
	,\left[ \phi \theta _{2}\right] ,\dots,\left[ \phi \theta _{s}\right]
	\right\rangle $. Then $\phi W\in G_{s}\left( {\rm H}^{2}\left( {\bf A},{\mathbb C} \right) \right) $. We denote the orbit of $W\in G_{s}\left(
	{\rm H}^{2}\left( {\bf A},{\mathbb C}\right) \right) $ under the action of $\aut(\mathbf{A})$ by $\operatorname{Orb}(W)$. Since given
	\[
	W_{1}=\left\langle \left[ \theta _{1}\right] ,\left[ \theta _{2}\right] ,\dots,
	\left[ \theta _{s}\right] \right\rangle ,W_{2}=\left\langle \left[ \vartheta
	_{1}\right] ,\left[ \vartheta _{2}\right] ,\dots,\left[ \vartheta _{s}\right]
	\right\rangle \in G_{s}\left( {\rm H}^{2}\left( {\bf A},{\mathbb C}\right)
	\right),
	\]
	we easily have that in case $W_{1}=W_{2}$, then 
	$ \bigcap\limits_{i=1}^{s} \ann(\theta_{i})\cap \ann( {\bf A}) = \bigcap\limits_{i=1}^{s}
	\ann (\vartheta_{i})\cap \ann ( {\bf A}) $, and so we can introduce
	the set
	\[
	T_{s}(\mathbf{A}) =\left\{ W=\left\langle \left[ \theta _{1}\right] ,
	\left[ \theta _{2}\right] ,\dots,\left[ \theta _{s}\right] \right\rangle \in
	G_{s}\left( {\rm H}^{2}\left( {\bf A},{\mathbb C}\right) \right) : \bigcap\limits_{i=1}^{s}\ann(\theta _{i})\cap \ann(\mathbf{A}) =0\right\},
	\]
	which is stable under the action of $\aut(\mathbf{A})$.
	
	\
	
	Now, let ${\mathbb V}$ be an $s$-dimensional linear space and let us denote by
	$E\left( {\bf A},{\mathbb V}\right) $ the set of all {\it non-split $s$-dimensional central extensions} of ${\bf A}$ by
	${\mathbb V}$. We can write
	\[
	E\left( {\bf A},{\mathbb V}\right) =\left\{ {\bf A}_{\theta }:\theta \left( x,y\right) = \sum_{i=1}^{s}\theta _{i}\left( x,y\right) e_{i} \ \ \text{and} \ \ \left\langle \left[ \theta _{1}\right] ,\left[ \theta _{2}\right] ,\dots,
	\left[ \theta _{s}\right] \right\rangle \in T_{s}(\mathbf{A}) \right\} .
	\]
	We also have the next result, which can be proved as in \cite[Lemma 17]{hac16}.
	
	\begin{lemma}
		Let ${\bf A}_{\theta },{\bf A}_{\vartheta }\in E\left( {\bf A},{\mathbb V}\right) $. Suppose that $\theta \left( x,y\right) =  \displaystyle \sum_{i=1}^{s}
		\theta _{i}\left( x,y\right) e_{i}$ and $\vartheta \left( x,y\right) =
		\displaystyle \sum_{i=1}^{s} \vartheta _{i}\left( x,y\right) e_{i}$.
		Then the Zinbiel algebras ${\bf A}_{\theta }$ and ${\bf A}_{\vartheta } $ are isomorphic
		if and only if \[\operatorname{Orb}\left\langle \left[ \theta _{1}\right] ,
		\left[ \theta _{2}\right] ,\dots,\left[ \theta _{s}\right] \right\rangle =
		\operatorname{Orb}\left\langle \left[ \vartheta _{1}\right] ,\left[ \vartheta
		_{2}\right] ,\dots,\left[ \vartheta _{s}\right] \right\rangle .\]
	\end{lemma}
	
	From here, there exists a one-to-one correspondence between the set of $\aut(\mathbf{A})$-orbits on $T_{s}\left( {\bf A}\right) $ and the set of
	isomorphism classes of $E\left( {\bf A},{\mathbb V}\right) $. Consequently we have a
	procedure that allows us, given the Zinbiel algebra $\mathbf{A}'$ of
	dimension $n-s$, to construct all non-split central extensions of $\mathbf{A}'$. This procedure would be:

	{\centerline {\textsl{Procedure}}}
	
	\begin{enumerate}
		\item For a given Zinbiel algebra $\mathbf{A}'$ of dimension $n-s $, determine ${\rm H}^{2}( \mathbf{A}',{\mathbb C}) $, $\ann(\mathbf{A}')$ and $\aut(\mathbf{A}')$.
		
		\item Determine the set of $\aut(\mathbf{A}')$-orbits on $T_{s}(\mathbf{A}') $.
		
		\item For each orbit, construct the Zinbiel algebra corresponding to a
		representative of it.
	\end{enumerate}
	
	\
	
	Finally, let us introduce some of notation. Let ${\bf A}$ be a Zinbiel algebra with
	a basis $e_{1},e_{2},\dots,e_{n}$. Then by $\Delta _{i,j}$\ we will denote the
	bilinear form
	$\Delta _{i,j} \colon {\bf A}\times {\bf A}\longrightarrow {\mathbb C}$
	with $\Delta _{i,j}\left( e_{l},e_{m}\right) = \delta_{il}\delta_{jm}$.
	Then the set $\left\{ \Delta_{i,j}:1\leq i, j\leq n\right\} $ is a basis for the linear space of
	the bilinear forms on ${\bf A}$. Then every $\theta \in
	{\rm Z}^{2}\left( {\bf A},{\mathbb C}\right) $ can be uniquely written as $
	\theta = \displaystyle \sum_{1\leq i,j\leq n} c_{ij}\Delta _{{i},{j}}$, where $
	c_{ij}\in {\mathbb C}$.

	\section{Central extension of filiform Zinbiel algebras}\label{S:fil}
	
	\begin{proposition} Let $F_n^1,F_n^2$ and $F_n^3$ be $n$-dimensional  filiform Zinbiel algebras defined in Theorem \ref{th-filiform}. Then:
		
		\begin{itemize}
			\item A basis of ${\rm Z}^{2}(F_n^k,\mathbb{C})$ is formed by the following cocycles
			\[\begin{array}{l}
			{\rm Z}^{2}(F_n^1,\mathbb{C})=\langle\Delta_{1,1},\Delta_{1,n},\Delta_{n,1},\Delta_{n,n},\sum\limits_{i=1}^{s-1}C_{s-1}^{i-1}\Delta_{i,s-i}; \ 3\leq s\leq n\rangle, \\
			{\rm Z}^{2}(F_n^k,\mathbb{C})=\langle\Delta_{1,1},\Delta_{1,n},\Delta_{n,1},\Delta_{n,n},\sum\limits_{i=1}^{s-1}C_{s-1}^{i-1}\Delta_{i,s-i}; \ 3\leq s\leq n-1\rangle, k= 2, 3.
			\end{array}\]
			\item A basis of ${\rm B}^2(F_n^k,\mathbb{C})$  is formed by the following coboundaries
			\[\begin{array}{l}
			{\rm B}^2(F_n^1,\mathbb{C})=\langle\Delta_{1,1},\sum\limits_{i=1}^{s-1}C_{s-1}^{i-1}\Delta_{i,s-i}, \ 3\leq s\leq n-1\rangle, \\
			{\rm B}^2(F_n^2,\mathbb{C})=\langle\Delta_{1,1},\sum\limits_{i=1}^{s-1}C_{s-1}^{i-1}\Delta_{i,s-i}, \ 3\leq s\leq n-2,\sum\limits_{i=1}^{n-2}C_{n-2}^{i-1}\Delta_{i,n-1-i}+\Delta_{n,1}\rangle,\\
			{\rm B}^2(F_n^3,\mathbb{C})=\langle\Delta_{1,1},\sum\limits_{i=1}^{s-1}C_{s-1}^{i-1}\Delta_{i,s-i}, \ 3\leq s\leq n-2,\sum\limits_{i=1}^{n-2}C_{n-2}^{i-1}\Delta_{i,n-1-i}+\Delta_{n,n}\rangle.
			\end{array}\]
			\item A basis of ${\rm H}^2(F_n^k,\mathbb{C})$ is formed by the following cocycles
			\[\begin{array}{l}
			{\rm H}^2(F_n^1,\mathbb{C})=\langle[\Delta_{1,n}], [\Delta_{n,1}], [\Delta_{n,n}], [\sum\limits_{i=1}^{n-1}C_{n-1}^{i-1}\Delta_{i,n-i}] \rangle, \\
			{\rm H}^2(F_n^k,\mathbb{C})=\langle [\Delta_{1,n}],[\Delta_{n,1}],[\Delta_{n,n}] \rangle, \quad k=2, 3.
			
			\end{array}\]
		\end{itemize}
	\end{proposition}
	\begin{proof}
		The proof follows directly from the definition of a cocycle.
	\end{proof}
	
	\begin{proposition}
		Let $\phi^n_{k}\in \aut(F_n^k)$. Then
		
		\begin{longtable}{cc}
$			\phi_{1}^n= \begin{pmatrix}
				a_{1,1} &0  & 0 &\dots  &  0& 0 \\
				a_{2,1} &a_{1,1}^2  & 0 & \dots & 0 & 0 \\
				a_{3,1}& * & a_{1,1}^3 &\dots  &  0& 0 \\
				\vdots & \vdots & \vdots &\ddots  & \vdots & \vdots \\
				a_{n-1,1} &*  &*  &  & a_{1,1}^{n-1} & a_{n-1,n} \\
				a_{n,1} & 0 & 0 & \dots & 0 & a_{n,n} \\
			\end{pmatrix},$  &
$			\phi_{2}^n=
			\begin{pmatrix}
				a_{1,1} &0  & 0 &\dots  &  0& 0 \\
				a_{2,1} &a_{1,1}^2  & 0 & \dots & 0 & 0 \\
				a_{3,1}& * & a_{1,1}^3 &\dots  &  0& 0 \\
				\vdots & \vdots & \vdots &\ddots  & \vdots & \vdots \\
				a_{n-1,1} &*  &*  &  & a_{1,1}^{n-1} & a_{n-1,n} \\
				a_{n,1} & 0 & 0 & \dots & 0 & a_{1,1}^{n-2} \\
			\end{pmatrix},$ \\
\multicolumn{2}{c}{$			\phi_{3}^n=
			\begin{pmatrix}
				a_{1,1} &0  & 0 &\dots  &  0& 0 \\
				a_{2,1} &a_{1,1}^2  & 0 & \dots & 0 & 0 \\
				a_{3,1}& * & a_{1,1}^3 &\dots  &  0& 0 \\
				\vdots & \vdots & \vdots &\ddots  & \vdots & \vdots \\
				a_{n-1,1} &*  &*  &  & a_{1,1}^{n-1} & a_{n-1,n} \\
				a_{n,1} & 0 & 0 & \dots & 0 & a_{1,1}^{(n-1)/2} \\
			\end{pmatrix} $}
		\end{longtable}
	\end{proposition}
	
	\subsection{Central extensions of $F_n^1$}

	Let us denote
	\[\nabla_1=[\Delta_{1,n}], \   \nabla_2= [\Delta_{n,1}], \ \nabla_3= [\Delta_{n,n}], \  \nabla_4= [\sum\limits_{j=1}^{n-1}C_{n-1}^{j-1}\Delta_{j,n-j}]\]
	and $x=a_{1,1},y=a_{n,n}, z=a_{n-1,n},w=a_{n,1}$.
	Since
	{\tiny \[\left(
		\begin{array}{ccccc}
		\ast & \dots & \ast & C_{n-1}^0\alpha_4^\prime & \alpha_1^\prime \\
		\ast & \dots & C_{n-1}^{1}\alpha_4^\prime & 0 & 0 \\
		\vdots & \ldots & \vdots & \vdots & \vdots \\
		C_{n-1}^{n-2}  \alpha_4^\prime & \dots & 0 & 0 & 0 \\
		\alpha_2^\prime & \dots & 0 & 0 & \alpha_3^\prime \\
		\end{array}
		\right)=(\phi_{1}^n)^T
		\left(
		\begin{array}{ccccccc}
		0 & 0 &  0 & \cdots & 0 &C_{n-1}^0\alpha_4 & \alpha_1 \\
		0 & 0 &  0 & \cdots &C_{n-1}^1\alpha_4  &0 & 0 \\
		0 & 0 &  0  &\vdots  & \cdot^{\cdot^{\cdot}}  & 0 & 0 \\
		\vdots &\vdots & .^{{{.}^.}} & C_{n-1}^{n-1-i}\alpha_4 & \cdot^{\cdot^{\cdot}}  & \vdots & \vdots \\
		\vdots & 0 & \cdot^{\cdot^{\cdot}} &0&
		\vdots& \vdots & \vdots  \\
		0 & C_{n-1}^{n-3}\alpha_4 & 0 &\vdots&
		\vdots& \vdots & \vdots  \\
		C_{n-1}^{n-2}\alpha_4 & 0 &0&
		\cdots& 0 & 0 &0 \\
		
		\alpha_2 &0 &0 &\cdots & 0 & 0 & \alpha_3 \\
		\end{array}
		\right)\phi_{1}^n,\]}
	for any
	$\theta=\alpha_1 \nabla_1 +\alpha_2 \nabla_2 +\alpha_3 \nabla_3+\alpha_4 \nabla_4$, we have
	the action of the automorphism group on the subspace $\langle \theta \rangle$ as
	\[\Big\langle
	(\alpha_1 xy+\alpha_3 yw+\alpha_4 xz)  \nabla_1 + (\alpha_2 xy+\alpha_3 yw+(n-1)\alpha_4  xz) \nabla_2 +
	\alpha_3 y^2 \nabla_3+  \alpha_4 x^n \nabla_4 \Big\rangle.\]
	
	\subsubsection{$1$-dimensional central extensions of $F_n^1$}
	Let us consider the following cases:
	
	\begin{enumerate}
		\item if $\alpha_1\neq 0, \alpha_2=\alpha_3=\alpha_4=0,$ then by choosing $x= 1, y=1/\alpha_1,$ we have the representative $\langle \nabla_1 
		\rangle.$
		
		\item if     $\alpha_2\neq 0, \alpha_3=\alpha_4=0,$ then by choosing $x= 1, y=1/\alpha_2, \alpha= \alpha_1 / \alpha_2,$ we have the family of representatives $\langle\alpha \nabla_1+\nabla_2\rangle.$
		
		\item if     $\alpha_1=\alpha_2, \alpha_3\neq 0, \alpha_4=0,$ then by choosing $y=1/\sqrt{\alpha_3}, w= - \alpha_2 / \alpha_3, x=1,$ we have the representative $\langle \nabla_3 \rangle.$
		
		\item if     $ \alpha_1 \neq \alpha_2, \alpha_3\neq 0, \alpha_4=0,$ then by choosing $x=\frac{\sqrt{\alpha_3}}{\alpha_1-\alpha_2}, y=\frac{1}{\sqrt{\alpha_3}},w=  \frac{\alpha_2}{\sqrt{\alpha_3}(\alpha_2-\alpha_1)},$ we have the representative $\langle \nabla_1+ \nabla_3 \rangle.$
		
		\item if     $(n-1)\alpha_1=\alpha_2, \alpha_3=0,  \alpha_4\neq0,$ then by choosing 
		$x =  1/\sqrt[n]{\alpha_4}, y=1, z= - \alpha_1 / \alpha_4,$ we have the representative $\langle \nabla_4 \rangle.$
		
		\item if     $(n-1)\alpha_1\neq \alpha_2, \alpha_3=0, \alpha_4\neq 0,$ then by choosing 
		$x= 1/ \sqrt[n]{\alpha_4}, y=\frac{\sqrt[n]{\alpha_4}}{\alpha_2-(n-1)\alpha_1}, z=-\frac{\sqrt[n]{\alpha_4}}{\alpha_2-(n-1)\alpha_1},$ we have the representative $\langle \nabla_2+ \nabla_4 \rangle.$
		
		\item  if    $\alpha_3\neq 0, \alpha_4\neq 0,$ then by choosing 
		$x =  1/\sqrt[n]{\alpha_4}, y=1/\sqrt{\alpha_3},
		z=  \frac{\alpha_1-\alpha_2}{(n-2) \sqrt{\alpha_3}\alpha_4}, 
		w=  \frac{\alpha_2-(n-1)\alpha_1}{(n-2)\sqrt[n]{\alpha_4} \alpha_3},$ we have the representative $\langle \nabla_3 +\nabla_4\rangle.$
		
	\end{enumerate}

	It is easy to verify that all previous orbits  are different, and so we obtain
\begin{longtable}{ccl}
$T_1(F_n^1)$&$=$&$
		\operatorname{Orb}\langle \nabla_1  \rangle \cup
		\operatorname{Orb}\langle \alpha\nabla_1+ \nabla_2  \rangle \cup
		\operatorname{Orb}\langle \nabla_3  \rangle\cup
		\operatorname{Orb}\langle \nabla_1+ \nabla_3 \rangle\cup$\\
&&$		\operatorname{Orb}\langle \nabla_4 \rangle   
	 
		\cup
		\operatorname{Orb}\langle \nabla_2+ \nabla_4 \rangle 
		\cup
		\operatorname{Orb}\langle \nabla_3+ \nabla_4 \rangle .$
	\end{longtable}
	
	\subsubsection{$2$-dimensional central extensions of $F_n^1$}
	We may assume that a $2$-dimensional subspace is generated by
	\begin{align*}
		\theta_1 & = \alpha_1 \nabla_1+ \alpha_2 \nabla_2+\alpha_3 \nabla_3 +\alpha_4 \nabla_4,\\
		\theta_2 & = \beta_1 \nabla_1+ \beta_2 \nabla_2 +\beta_3 \nabla_3.
	\end{align*}
	
	Then we have the six following cases:
	
	\begin{enumerate}

		\item if $\alpha_4\neq 0, \beta_3\neq 0,$ then we can suppose that $\alpha_3=0.$ Now
		
		\begin{enumerate}
			\item for $(n-1)\alpha_1 \neq \alpha_2, \beta_1 \neq \beta_2,$  by choosing
			$x=\Big(\frac{(\alpha_2-(n-1)\alpha_1)(\beta_2-\beta_1)}{\alpha_4}\Big)^{1/(n-2)},$
			$y=\frac{\beta_2-\beta_1}{\beta_3}x,$
			$z=\frac{\alpha_1(\beta_1-\beta_2)}{\alpha_4\beta_3}x,$ 
			$w=-\beta_1 x/ \beta_3,$
			we have the representative $\langle \nabla_2+\nabla_3, \nabla_2+\nabla_4 \rangle.$

			\item for $(n-1)\alpha_1 \neq \alpha_2, \beta_1 = \beta_2,$  by choosing
			$x=\Big(\frac{\alpha_2-(n-1)\alpha_1}{\alpha_4\sqrt{\beta_3}}\Big)^{1/(n-1)},$
			$y=1/\sqrt{\beta_3},$
			$z=-\frac{\alpha_1}{\alpha_4\sqrt{\beta_3}},$ 
			$w=-\beta_1 x/ \beta_3,$
			we have the representative $\langle \nabla_3, \nabla_2+\nabla_4 \rangle.$

			\item for $(n-1)\alpha_1 = \alpha_2, \beta_1 \neq \beta_2,$  by choosing
			$x=1/\sqrt[n]{\alpha_4},$
			$y=\frac{\beta_2-\beta_1}{\beta_3}x,$
			$z=\frac{\alpha_1(\beta_1-\beta_2)}{\alpha_4\beta_3}x,$ 
			$w=-\beta_1 x/ \beta_3,$
			we have the representative $\langle \nabla_2+\nabla_3, \nabla_4 \rangle.$
			
			\item for $(n-1)\alpha_1 = \alpha_2, \beta_1 = \beta_2,$  by choosing
			$x=1/\sqrt[n]{\alpha_4},$
			$y=1/\sqrt{\beta_3},$
			$z=-\alpha_1y/\alpha_4,$ 
			$w=-\beta_1 x/ \beta_3,$
			we have the representative $\langle \nabla_3, \nabla_4 \rangle.$

		\end{enumerate}

		\item if $\alpha_4\neq 0, \beta_3= 0, \beta_2\neq 0,$ then we can suppose that $\alpha_2=0.$ Now
		
		\begin{enumerate} 
			\item for $\alpha_3 \neq0,$ by choosing 
			$x =  1/\sqrt[n]{\alpha_4}, y= 1/\sqrt{\alpha_3},
			z=  \frac{\alpha_1}{(n-2) \sqrt{\alpha_3}\alpha_4}, 
			w=  -\frac{(n-1)\alpha_1}{(n-2)\sqrt[n]{\alpha_4} \alpha_3},$ and $\alpha=\beta_1/\beta_2$ we have the family of representatives $\langle \alpha \nabla_1+\nabla_2, \nabla_3 +\nabla_4\rangle.$

			\item for $\alpha_3=0, (n-1)\beta_1\neq \beta_2,$ then
			by choosing $x =  1/\sqrt[n]{\alpha_4}, y=1, z= - \frac{\alpha_1\beta_2} {\alpha_4((n-1)\beta_1-\beta_2)}$ and $\alpha=\beta_1/\beta_2,$ we have the family of representatives $\langle \alpha \nabla_1+\nabla_2,  \nabla_4 \rangle_{\alpha \neq \frac 1 {n-1}}.$

			\item for $\alpha_3=0, (n-1)\beta_1=\beta_2$ and $\alpha_1=0,$ by choosing 
			$x =  1 /\sqrt[n]{\alpha_4}, y=1, z= 0,$ we have the representative $\langle \frac 1 {n-1}\nabla_1+\nabla_2,  \nabla_4 \rangle.$
			
			\item for $\alpha_3=0, (n-1)\beta_1=\beta_2$ and $\alpha_1\neq 0,$ by choosing     
			$x= 1/\sqrt[n]{\alpha_4}, y=-\frac{\sqrt[n]{\alpha_4}}{(n-1)\alpha_1}, z=\frac{\sqrt[n]{\alpha_4}}{(n-1)\alpha_4},$ we have the representative $\langle \frac 1 {n-1} \nabla_1+\nabla_2, \nabla_2+ \nabla_4 \rangle.$

		\end{enumerate}
		
		\item if $\alpha_4\neq 0, \beta_3=\beta_2=0, \beta_1\neq 0,$ then 
		
		\begin{enumerate}
			\item for $\alpha_3 \neq 0,$ by choosing 
			$x =  1 / \sqrt[n]{\alpha_4}, y= 1/\sqrt{\alpha_3},
			z=  \frac{\alpha_1-\alpha_2}{(n-2) \sqrt{\alpha_3}\alpha_4}, 
			w=  \frac{\alpha_2-(n-1)\alpha_1}{(n-2)\sqrt[n]{\alpha_4} \alpha_3},$ we have the representative $\langle  \nabla_1, \nabla_3 +\nabla_4\rangle.$

			\item for $\alpha_3=0,$ after a linear combination of $\theta_1$ and $\theta_2$ we can suppose that         $(n-1)\alpha_1=\alpha_2,$
			by choosing $x =  1/\sqrt[n]{\alpha_4}, y=1, z= - \alpha_1 / \alpha_4,$  we have the representative $\langle \nabla_1,  \nabla_4 \rangle.$
		\end{enumerate}

		\item if $\alpha_4=0, \alpha_3\neq 0, \beta_2\neq 0,$ then 
		
		\begin{enumerate}

			\item for $\beta_1 \neq \beta_2,$ after a linear combination of $\theta_1$ and $\theta_2$ we can suppose that  $\alpha_1=\alpha_2,$  by choosing $y=1/\sqrt{\alpha_3}, w= - \alpha_2 / \alpha_3, x=1$ and $\alpha=\beta_1/\beta_2$ we have the family of representatives $\langle \alpha \nabla_1+\nabla_2, \nabla_3 \rangle_{\alpha\neq 1}.$

			\item for $\beta_1=\beta_2, \alpha_1=\alpha_2,$ after a linear combination of $\theta_1$ and $\theta_2$ we have the representative $\langle \nabla_1+\nabla_2, \nabla_3 \rangle.$
			
			\item for $\beta_1=\beta_2, \alpha_1 \neq \alpha_2,$ by choosing $x=\frac{\sqrt{\alpha_3}}{\alpha_1-\alpha_2}, y= 1 / \sqrt{\alpha_3},
			w=\frac{\alpha_2}{\sqrt{\alpha_3}(\alpha_2-\alpha_1)},$ we have the representative $\langle \nabla_1+\nabla_2, \nabla_1+ \nabla_3 \rangle.$

		\end{enumerate}

		\item if $\alpha_4=0, \alpha_3\neq 0, \beta_2=0, \beta_1 \neq 0,$ then 
		after a linear combination of $\theta_1$ and $\theta_2$ we can suppose that  $\alpha_1=\alpha_2,$  by choosing $y=1/\sqrt{\alpha_3}, w= - \alpha_2 / \alpha_3, x=1$ and $\alpha=\beta_1/\beta_2$ we have the  representative $\langle \nabla_1, \nabla_3 \rangle.$

		\item if $\alpha_3=\alpha_4=0, \beta_3=0$, then we have the representative $\langle\nabla_1,\nabla_2\rangle.$

	\end{enumerate}

	It is easy to verify that all previous orbits   are different, and so we obtain
\begin{longtable}{ccl}
$T_2(F_n^1)$ & $=$&$
\operatorname{Orb} \langle   \nabla_1, \nabla_2 \rangle \cup
\operatorname{Orb} \langle   \nabla_1, \nabla_3 \rangle \cup
\operatorname{Orb} \langle   \nabla_1, \nabla_3+\nabla_4 \rangle \cup
\operatorname{Orb} \langle   \nabla_1, \nabla_4 \rangle \cup$\\
&&$
\operatorname{Orb} \langle   \frac 1 {n-1} \nabla_1+\nabla_2, \nabla_2+\nabla_4 \rangle \cup 
\operatorname{Orb} \langle   \nabla_1+\nabla_2, \nabla_1+\nabla_3  \rangle \cup 
\operatorname{Orb} \langle   \alpha \nabla_1+\nabla_2, \nabla_3 \rangle \cup $\\
&&$
\operatorname{Orb}  \langle   \alpha \nabla_1+\nabla_2, \nabla_3+\nabla_4 \rangle \cup		\operatorname{Orb} \langle   \alpha \nabla_1+\nabla_2, \nabla_4 \rangle \cup
		\operatorname{Orb} \langle   \nabla_2+\nabla_3, \nabla_2+\nabla_4 \rangle \cup$\\
&&$
		\operatorname{Orb} \langle   \nabla_2+\nabla_3, \nabla_4 \rangle \cup
		\operatorname{Orb} \langle   \nabla_3, \nabla_2+\nabla_4 \rangle \cup
		\operatorname{Orb} \langle   \nabla_3, \nabla_4 \rangle.$
	\end{longtable}

	\subsubsection{$3$-dimensional central extensions of $F_n^1$}
	We may assume that a $3$-dimensional subspace is generated by
	\begin{align*}
		\theta_1 &= \alpha_1 \nabla_1+ \alpha_2 \nabla_2+\alpha_3 \nabla_3 +\alpha_4 \nabla_4,\\
		\theta_2 &= \beta_1 \nabla_1+ \beta_2 \nabla_2 +\beta_3 \nabla_3,\\
		\theta_3 & = \gamma_1 \nabla_1+ \gamma_2 \nabla_2.
	\end{align*}
	
	Then we have the  following cases:

	\begin{enumerate}
		\item if $\alpha_4\neq 0, \beta_3 \neq0, \gamma_2\neq 0,$ then we can suppose that $\alpha_2=0,$ $\alpha_3=0,$ $\beta_2=0$ and 
		\begin{enumerate}
			\item for $\gamma_1\neq\gamma_2,$ $(n-1)\gamma_1\neq\gamma_2,$ then by choosing
			$x=1/\sqrt[n]{\alpha_4},$
			$y=1/\sqrt{\beta_3},$
			$z=\frac{\alpha_1\gamma_2 y} {\alpha_4((n-1)\gamma_1-\gamma_2)},$ 
			$w=\frac{\beta_1 \gamma_2 x} {\alpha_4(\gamma_1-\gamma_2)},$
			we have the family of representatives $\langle \alpha \nabla_1+\nabla_2, \nabla_3, \nabla_4 \rangle_{\alpha\not\in \{ 1, \frac 1 {n-1}\}}.$

			\item for $\gamma_1=\gamma_2,$ then 
				\begin{enumerate}
				\item for $\beta_1 \neq 0,$  by choosing
				$x=1/\sqrt[n]{\alpha_4},$
				$y=\frac{\beta_1x}{\beta_3},$
				$z=\frac{\alpha_1y}{(n-2)\alpha_4},$ 
				$w=0,$
				we have the representative $\langle \nabla_1+\nabla_2,\nabla_1+\nabla_3, \nabla_4 \rangle.$
				
				\item for $\beta_1 = 0,$  by choosing
				$x=1/\sqrt[n]{\alpha_4},$
				$y=1/\sqrt{\beta_3},$
				$z=\frac{\alpha_1y}{(n-2)\alpha_4},$ 
				$w=0,$
				we have the representative $\langle \nabla_1+\nabla_2,\nabla_3, \nabla_4 \rangle.$
			\end{enumerate}
		\item for $(n-1)\gamma_1=\gamma_2,$ then 
				\begin{enumerate}
				\item for $\alpha_1 \neq 0,$  by choosing
				$y=1/\sqrt{\beta_3},$
				$z=-\frac{\alpha_1y}{\alpha_4},$ 
				$x=\sqrt[n-1]{(n-1)z},$
				$w=-\frac{(n-1)\beta_1 x} {(n-2)\beta_3},$
				we have the representative $\langle \frac 1 {n-1} \nabla_1+\nabla_2,\nabla_3, \nabla_2+\nabla_4 \rangle.$
				
				\item for $\alpha_1 = 0,$  by choosing
				$x=1/\sqrt[n]{\alpha_4},$
				$y=1/\sqrt{\beta_3},$
				$z=0,$ 
				$w=-\frac{(n-1)\beta_1 x} {(n-2)\beta_3},$
				we have the representative $\langle \frac 1 {n-1}\nabla_1+\nabla_2,\nabla_3, \nabla_4 \rangle.$
			\end{enumerate}

		\end{enumerate}

		\item if $\alpha_4\neq 0, \beta_3 \neq0, \gamma_2= 0, \gamma_1\neq 0,$ then we can suppose that $\alpha_3=0$ and after a linear combination of $\theta_1, \theta_2, \theta_3$ we can suppose that  $(n-1)\alpha_1 = \alpha_2, \beta_1 = \beta_2.$   By choosing
		$x=1/\sqrt[n]{\alpha_4},$
		$y=1/\sqrt{\beta_3},$
		$z=-\alpha_1y/\alpha_4,$ 
		$w=-\beta_1 x/ \beta_3,$
		we have the representative $\langle  \nabla_1, \nabla_3, \nabla_4 \rangle.$

		\item if $\alpha_4\neq 0, \beta_3 =0, \beta_2\neq 0, \gamma_2=0, \gamma_1\neq 0,$ then
		and after a linear combination of $\theta_1, \theta_2, \theta_3$ we can suppose that $\alpha_1=\alpha_2=\beta_1=0.$ Now
		\begin{enumerate}
			\item for $\alpha_3\neq 0,$ by choosing $y=1/\sqrt{\alpha_3}, x=1/\sqrt[n]{\alpha_4}$ we have the representative $\langle \nabla_1, \nabla_2, \nabla_3+\nabla_4 \rangle.$
			\item for $\alpha_3=0,$ we have the representative  $\langle \nabla_1, \nabla_2, \nabla_4 \rangle.$
		\end{enumerate}
		
		\item if $\alpha_4= 0, \beta_3 = 0, \gamma_2= 0,$ then we have the representative $\langle \nabla_1, \nabla_2, \nabla_3 \rangle.$

	\end{enumerate}

	\

	It is easy to verify that all previous orbits  are different, and so we obtain
\begin{longtable}{ccl}
$T_3(F_n^1)$ & $=$&$
		\operatorname{Orb}\langle \nabla_1, \nabla_2, \nabla_3\rangle \cup
		\operatorname{Orb}\langle \nabla_1, \nabla_2, \nabla_3+\nabla_4\rangle \cup
		\operatorname{Orb}\langle \nabla_1, \nabla_2, \nabla_4\rangle \cup$\\
&&$ \operatorname{Orb}\langle \nabla_1+ \nabla_2, \nabla_1+\nabla_3, \nabla_4 \rangle \cup
\operatorname{Orb}\langle \frac 1 {n-1}\nabla_1+ \nabla_2, \nabla_3, \nabla_2+\nabla_4 \rangle \cup$\\
&&$
		\operatorname{Orb}\langle \alpha \nabla_1+\nabla_2, \nabla_3, \nabla_4 \rangle \cup
		\operatorname{Orb}\langle \nabla_1, \nabla_3,\nabla_4\rangle.$
	\end{longtable}
	
	\subsubsection{$4$-dimensional central extensions of $F_n^1$}
	There is only one $4$-dimensional non-split central extension of the algebra $F_n^1$.
	It is defined by $\langle \nabla_1, \nabla_2, \nabla_3, \nabla_4 \rangle$.
	
	\subsubsection{Non-split central extensions of $F_n^1$}
	So we have the next theorem
	
	\begin{theorem}
		An arbitrary non-split central extension of the algebra $F_n^1$ is isomorphic to one of the following pairwise non-isomorphic algebras
		
		\begin{itemize}
			\item one-dimensional central extensions:
			$$\mu_1^{n+1},\ \mu_2^{n+1}(\alpha),\ \mu_{3}^{n+1},\ \mu_{4}^{n+1},\ F_{n+1}^1,\ F_{n+1}^2, \ F_{n+1}^3$$

		\item two-dimensional central extensions:
		$$\mu_5^{n+2},\ \mu_6^{n+2},\ \mu_{7}^{n+2},\
		\mu_1^{n+2},\ \mu_{8}^{n+2}, \ \mu_{9}^{n+2},\ \mu_{10}^{n+2}(\alpha), \ \mu_{11}^{n+2}(\alpha),\
		\mu_2^{n+2}(\alpha),\ \mu_{12}^{n+2},\ \mu_4^{n+2},\ \mu_{13}^{n+2}, \  \mu_3^{n+2}$$

	\item three-dimensional central extensions:
	$$\mu_{14}^{n+3},\ \mu_{15}^{n+3},\ \mu_{5}^{n+3},\ \mu_{9}^{n+3},\ \mu_{16}^{n+3},\ \mu_{10}^{n+3}(\alpha),\ \mu_{6}^{n+3}$$

\item four-dimensional central extensions:
$$\mu_{14}^{n+4}$$

\end{itemize}
with $\alpha \in \mathbb{C}.$
\end{theorem}

\subsection{Central extensions of $F_n^2$}

Let us denote
\[ \nabla_1= [\Delta_{1,n}], \ \ \nabla_2= [\Delta_{n,1}], \ \ \nabla_3= [\Delta_{n,n}] \]
and $x=a_{1,1},w=a_{n,1}$.  Let $\theta=\alpha_1 \nabla_1 +\alpha_2 \nabla_2+\alpha_3\nabla_3$. Then by
\[\left(
\begin{array}{ccccc}
\ast & \dots & 0 & 0 & \alpha_1^\prime \\
0 & \dots & 0 & 0 & 0 \\
\vdots & \ldots & \vdots & \vdots & \vdots \\
0 & \dots & 0 & 0 & 0 \\
\alpha_2^\prime & \dots & 0 & 0 & \alpha_3^\prime \\
\end{array}
\right)=(\phi_{2}^n)^T\left(
\begin{array}{ccccc}
0 & \dots & 0 & 0 & \alpha_1 \\
0 & \dots & 0 & 0 & 0 \\
\vdots & \ldots & \vdots & \vdots & \vdots \\
0 & \dots & 0 & 0 & 0 \\
\alpha_2 & \dots & 0 & 0 & \alpha_3 \\
\end{array}
\right)\phi_{2}^n,\]
we have the action of the automorphism group on the subspace $\langle \theta \rangle$ as
\[\Big\langle
x^{n-2}(x\alpha_1+w\alpha_3)\nabla_1 + x^{n-2}(x\alpha_2+w\alpha_3)\nabla_2 +  x^{2n-4}\alpha_3 \nabla_3 \Big\rangle.\]

\subsubsection{$1$-dimensional central extensions of $F_n^2$}
Let us consider the following cases:

\begin{enumerate}

\item if $\alpha_3 =0,$ then 
\begin{enumerate}
\item for  $\alpha_2=0, \alpha_1\neq 0,$ we have the representative $\langle \nabla_1 \rangle.$

\item for $\alpha_2\neq0,$  by choosing $x=\alpha_2^{-1/(n-1)}$ and  $\alpha= \alpha_1 / \alpha_2,$ we have the family of representatives $\langle \alpha\nabla_1+ \nabla_2 \rangle.$ 
\end{enumerate}

\item if $\alpha_3 \neq 0,$ then

\begin{enumerate}
\item for $\alpha_1 \neq \alpha_2,$ by choosing $x=(\frac{\alpha_2-\alpha_1}{\alpha_3})^{1/(n-3)}, w=-\frac{x\alpha_1}{\alpha_3}$ we have the representative $\langle \nabla_2+ \nabla_3 \rangle$.
\item for $\alpha_1=\alpha_2,$ 
by choosing $w=-\frac{x\alpha_1}{\alpha_3}$ we have the representative $\langle  \nabla_3 \rangle.$
\end{enumerate}

\end{enumerate}

It is easy to verify that all previous  orbits   are different, and so we obtain
\begin{longtable}{ccl}
$T_1(F_n^2)$&$=$&$\operatorname{Orb} \langle \nabla_1 \rangle \cup
\operatorname{Orb} \langle \alpha\nabla_1 +\nabla_2  \rangle \cup
\operatorname{Orb} \langle \nabla_2+ \nabla_3 \rangle \cup \operatorname{Orb}\langle \nabla_3  \rangle.$
\end{longtable}

\subsubsection{$2$-dimensional central extensions of $F_n^2$}

We may assume that a $2$-dimensional subspace is generated by
\begin{align*}
\theta_1 & = \alpha_1 \nabla_1+ \alpha_2 \nabla_2+\alpha_3 \nabla_3, \\
\theta_2 & = \beta_1 \nabla_1+ \beta_2 \nabla_2.
\end{align*}

We consider the following cases:

\begin{enumerate}
\item if $\alpha_3\neq0$ and $\beta_1\neq \beta_2,$ then after a linear combination of $\theta_1$ and $\theta_2$ we can suppose that $\alpha_1=\alpha_2.$
Now,
\begin{enumerate}
\item for $\beta_2 \neq 0,$ by choosing
$x=\beta_2^{-1/(n-1)}, w=-\frac{x\alpha_1}{\alpha_3}$ and $\alpha=\beta_1/\beta_2$ 
we have the family of  respresentatives $\langle \alpha \nabla_1+\nabla_2, \nabla_3 \rangle_{\alpha\neq 1}.$

\item for $\beta_2 = 0,$ by choosing
$x=\beta_1^{-1/(n-1)}, w=-\frac{x\alpha_1}{\alpha_3},$ 
we have the respresentative $\langle  \nabla_1, \nabla_3 \rangle.$

\end{enumerate}

\item if $\alpha_3 \neq 0$ and $\beta_1=\beta_2,$ then 

\begin{enumerate}

\item for $\alpha_1 \neq \alpha_2,$ by choosing 
$x=(\frac{\alpha_1-\alpha_2}{\alpha_3})^{1/(n-1)},  w=-\frac{x\alpha_2}{\alpha_3}$ we have the representative $\langle \nabla_1+ \nabla_2, \nabla_1+\nabla_3 \rangle$.

\item for $\alpha_1=\alpha_2,$ after a linear combination of $\theta_1$ and $\theta_2$ we have the representative $\langle \nabla_1+\nabla_2, \nabla_3 \rangle.$

\end{enumerate}

\item if $\alpha_3=0,$
then we have the representative $\langle \nabla_1,\nabla_2\rangle.$
\end{enumerate}

It is easy to verify that all previous orbits  are different, and so we obtain
\begin{longtable}{ccl}
$T_2(F_n^2)$&$=$&$
\operatorname{Orb} \langle \nabla_1, \nabla_2  \rangle \cup
\operatorname{Orb} \langle \nabla_1, \nabla_3  \rangle \cup
\operatorname{Orb} \langle \nabla_1+\nabla_2, \nabla_1+\nabla_3  \rangle \cup
\operatorname{Orb} \langle \alpha \nabla_1+\nabla_2, \nabla_3  \rangle. $
\end{longtable}

\subsubsection{$3$-dimensional central extensions of $F_n^2$}
There is only one $3$-dimensional non-split central extension of the algebra $F_n^2$.
It is defined by $\langle \nabla_1, \nabla_2, \nabla_3 \rangle$.

\subsubsection{Non-split central extensions of $F_n^2$}
So we have the next result.

\begin{theorem}
An arbitrary non-split central extension of the algebra $F_n^2$ is isomorphic to one of the following pairwise non-isomorphic algebras

\begin{itemize}
\item one-dimensional central extensions:
$$\mu_1^{n+1},\ \mu_2^{n+1}(\alpha) \  \mbox{ with } \alpha\neq \frac{1}{n-3},\ \mu_{8}^{n+1},\ \mu_{12}^{n+1},\ \mu_{16}^{n+1}$$
\item two-dimensional central extensions:
$$\mu_{5}^{n+2},\ \mu_{6}^{n+2},\ \mu_{9}^{n+2},\ \mu_{10}^{n+2}(\alpha)\mbox{ with } \alpha\neq \frac{1}{n-4},\
\mu_{16}^{n+2}$$

\item three-dimensional central extensions:
$$\mu_{14}^{n+3}$$

\end{itemize}
with $\alpha \in \mathbb{C}.$
\end{theorem}

\subsection{Central extensions of $F_n^3$}

Let us denote
\[ \nabla_1= [\Delta_{1,n}], \ \ \nabla_2= [\Delta_{n,1}], \ \ \nabla_3= [\Delta_{n,n}] \]
and $x=a_{1,1},w=a_{n,1}$.  Let $\theta=\alpha_1 \nabla_1 +\alpha_2 \nabla_2+\alpha_3\nabla_3$. Then by

\[\left(
\begin{array}{ccccc}
\ast & \dots & 0 & 0 & \alpha_1^\prime \\
0 & \dots & 0 & 0 & 0 \\
\vdots & \ldots & \vdots & \vdots & \vdots \\
0 & \dots & 0 & 0 & 0 \\
\alpha_2^\prime & \dots & 0 & 0 & \alpha_3^\prime \\
\end{array}
\right)=(\phi_{3}^n)^T\left(
\begin{array}{ccccc}
0 & \dots & 0 & 0 & \alpha_1 \\
0 & \dots & 0 & 0 & 0 \\
\vdots & \ldots & \vdots & \vdots & \vdots \\
0 & \dots & 0 & 0 & 0 \\
\alpha_2 & \dots & 0 & 0 & \alpha_3 \\
\end{array}
\right)\phi_{3}^n,\]
we have the action of the automorphism group on the subspace $\langle \theta \rangle$ as
\[\Big\langle
x^{(n-1)/2}(x\alpha_1+w\alpha_3)\nabla_1 + x^{(n-1)/2}(x\alpha_2+w\alpha_3)\nabla_2 +  x^{n-1}\alpha_3 \nabla_3\Big\rangle.\]

\subsubsection{$1$-dimensional central extensions of $F_n^3$}
Let us consider the following cases:

\begin{enumerate}

\item if $\alpha_3 =0,$ then
\begin{enumerate}
\item for $\alpha_2=0,  \alpha_1\neq 0,$  by  choosing $x=\alpha_1^{-2/(n+1)},$
we have the representative  $\langle \nabla_1 \rangle.$

\item for $\alpha_2\neq0,$  by  choosing $x=\alpha_2^{-2/(n+1)}$ and $\alpha=\alpha_1/\alpha_2$
we have the family of  representatives  $\langle \alpha\nabla_1+ \nabla_2 \rangle.$
\end{enumerate}

\item  if $\alpha_3 \neq 0,$ then 
\begin{enumerate}
\item  for $\alpha_2 \neq \alpha_1,$ by choosing 
$x=(\frac{\alpha_2-\alpha_1}{\alpha_3})^{2/(n-3)},  w=-\frac{x\alpha_1}{\alpha_3}$ we have the representative $\langle \nabla_2+ \nabla_3 \rangle$.
\item for $\alpha_2=\alpha_1,$ by choosing $ w=-\frac{x\alpha_1}{\alpha_3}$ we have the representative $\langle  \nabla_3 \rangle$.
\end{enumerate}

\end{enumerate}

It is easy to verify that all previous  orbits   are different, and so we obtain
\begin{longtable}{ccl}
$T_1(F_n^3)$&$=$&$\operatorname{Orb} \langle \nabla_1  \rangle \cup
\operatorname{Orb} \langle \alpha\nabla_1 +\nabla_2  \rangle \cup
\operatorname{Orb} \langle \nabla_2+ \nabla_3  \rangle \cup
\operatorname{Orb} \langle \nabla_3  \rangle.$
\end{longtable}

\subsubsection{$2$-dimensional central extensions of $F_n^3$}

We may assume that a $2$-dimensional subspace is generated by
\begin{align*}
\theta_1 & = \alpha_1 \nabla_1+ \alpha_2 \nabla_2+\alpha_3 \nabla_3, \\
\theta_2 & = \beta_1 \nabla_1+ \beta_2 \nabla_2.
\end{align*}

We consider the following cases:

\begin{enumerate}
\item if $\alpha_3\neq0$ and $\beta_1\neq \beta_2,$ then after a linear combination of $\theta_1$ and $\theta_2$ we can suppose that $\alpha_1=\alpha_2.$
Now,
\begin{enumerate}
\item for $\beta_2 \neq 0,$ by choosing
$x=\beta_2^{-2/(n+1)}, w=-\frac{x\alpha_1}{\alpha_3}$ and $\alpha=\beta_1/\beta_2$ 
we have the family of respresentatives $\langle \alpha \nabla_1+\nabla_2, \nabla_3 \rangle_{\alpha\neq 1}.$

\item for $\beta_2 = 0,$ by choosing
$x=\beta_1^{-2/(n+1)}, w=-\frac{x\alpha_1}{\alpha_3},$ 
we have the respresentative $\langle  \nabla_1, \nabla_3 \rangle.$

\end{enumerate}

\item if $\alpha_3 \neq 0$ and $\beta_1=\beta_2,$ then 

\begin{enumerate}

\item for $\alpha_1 \neq \alpha_2,$ by choosing 
$x=(\frac{\alpha_1-\alpha_2}{\alpha_3})^{2/(n-3)},  w=-\frac{x\alpha_2}{\alpha_3}$ we have the representative $\langle \nabla_1+ \nabla_2, \nabla_1+\nabla_3 \rangle$.

\item for $\alpha_1=\alpha_2,$ after a linear combination of $\theta_1$ and $\theta_2$ we have the representative $\langle \nabla_1+\nabla_2, \nabla_3 \rangle.$

\end{enumerate}

\item if $\alpha_3=0,$
then we have the representative $\langle \nabla_1,\nabla_2\rangle.$
\end{enumerate}

It is easy to verify that all previous orbits  are different, and so we obtain
\begin{longtable}{ccl}
$T_2(F_n^3)$&$=$&$
\operatorname{Orb} \langle \nabla_1, \nabla_2  \rangle \cup
\operatorname{Orb} \langle \nabla_1, \nabla_3  \rangle \cup
\operatorname{Orb} \langle \nabla_1+\nabla_2, \nabla_1+\nabla_3  \rangle \cup
\operatorname{Orb} \langle \alpha \nabla_1+\nabla_2, \nabla_3.  \rangle $
\end{longtable}

\subsubsection{$3$-dimensional central extensions of $F_n^3$}
There is only one $3$-dimensional non-split central extension of the algebra $F_n^3$.
It is defined by $\langle \nabla_1, \nabla_2, \nabla_3 \rangle$.

\subsubsection{Non-split central extensions of $F_n^3$}
So we have the next theorem.

\begin{theorem}
An arbitrary non-split central extension of the algebra $F_n^3$ is isomorphic to one of the following pairwise non-isomorphic algebras
\begin{itemize}
\item one-dimensional central extensions:
$$\mu_7^{n+1},\  \mu_{11}^{n+1}(\alpha),\  \mu_{12}^{n+1},\  \mu_{3}^{n+1}$$

\item two-dimensional central extensions:
$$\mu_{15}^{n+2},\ \mu_{6}^{n+2},\ \mu_{9}^{n+2},\ \mu_{10}^{n+2}(\alpha)$$
\item three-dimensional central extensions:
$$\mu_{14}^{n+3}$$

\end{itemize}
with $\alpha \in \mathbb{C}.$
\end{theorem}

\section{Appendix: The list of the algebras}

{\tiny 
\begin{longtable}{lllllll}
$\mu_1^n$&$:$& $e_i\circ e_j=C_{i+j-1}^j,$&$2\leq i+j\leq n-2,$
&$e_1\circ e_n=e_{n-1},$& &\\[2mm]

$\mu_2^n(\alpha)$&$:$& $e_i\circ e_j=C_{i+j-1}^j,$&$2\leq i+j\leq n-2,$
&$e_1\circ e_n=\alpha e_{n-1},$&$e_{n}\circ e_1=e_{n-1},$&\\[2mm]

$\mu_3^n$&$:$& $e_i\circ e_j=C_{i+j-1}^j,$&$2\leq i+j\leq n-2,$
&$e_n\circ e_n=e_{n-1},$&&\\[2mm]

$\mu_4^n$&$:$& $e_i\circ e_j=C_{i+j-1}^j,$&$2\leq i+j\leq n-2,$
&$e_1\circ e_n= e_{n-1},$&$e_{n}\circ e_n=e_{n-1},$&\\[2mm]

$\mu_5^n$&$:$& $e_i\circ e_j=C_{i+j-1}^j,$&$2\leq i+j\leq n-3,$
&$e_1\circ e_n= e_{n-1},$&$e_{n}\circ e_1=e_{n-2},$&\\[2mm]

$\mu_6^n$&$:$& $e_i\circ e_j=C_{i+j-1}^j,$&$2\leq i+j\leq n-3,$
&$e_1\circ e_n= e_{n-1},$&$e_{n}\circ e_n=e_{n-2},$&\\[2mm]

$\mu_7^n$&$:$& $e_i\circ e_j=C_{i+j-1}^j,$&$2\leq i+j\leq n-2,$
&$e_1\circ e_n= e_{n-1},$&$e_{n}\circ e_n=e_{n-2},$&\\[2mm]

$\mu_8^n$&$:$& $e_i\circ e_j=C_{i+j-1}^j,$&$2\leq i+j\leq n-2,$
&$e_1\circ e_n=\frac{1}{n-3}e_{n-1},$&$e_{n}\circ e_1=e_{n-2}+e_{n-1},$&\\[2mm]

$\mu_9^n$&$:$& $e_i\circ e_j=C_{i+j-1}^j,$&$2\leq i+j\leq n-3,$
&$e_1\circ e_n=e_{n-2}+e_{n-1},$&$e_{n}\circ e_1=e_{n-1},$&$e_n\circ e_n=e_{n-2}$\\[2mm]

$\mu_{10}^n(\alpha)$&$:$& $e_i\circ e_j=C_{i+j-1}^j,$&$2\leq i+j\leq n-3,$
&$e_1\circ e_n=\alpha e_{n-1},$&$e_{n}\circ e_1=e_{n-1},$&$e_n\circ e_n=e_{n-2},$\\[2mm]

$\mu_{11}^n(\alpha)$&$:$& $e_i\circ e_j=C_{i+j-1}^j,$&$2\leq i+j\leq n-2,$
&$e_1\circ e_n=\alpha e_{n-1},$&$e_{n}\circ e_1=e_{n-1},$&$e_n\circ e_n=e_{n-2},$\\[2mm]

$\mu_{12}^n$&$:$& $e_i\circ e_j=C_{i+j-1}^j,$&$2\leq i+j\leq n-2,$
&&$e_{n}\circ e_1=e_{n-2}+e_{n-1},$&$e_n\circ e_n=e_{n-1},$\\[2mm]

$\mu_{13}^n$&$:$& $e_i\circ e_j=C_{i+j-1}^j,$&$2\leq i+j\leq n-2,$
&&$e_{n}\circ e_1=e_{n-2},$&$e_n\circ e_n=e_{n-1},$\\[2mm]

$\mu_{14}^n$&$:$& $e_i\circ e_j=C_{i+j-1}^j,$&$2\leq i+j\leq n-4,$
&$e_1\circ e_n=e_{n-2},$&$e_{n}\circ e_1=e_{n-1},$&$e_n\circ e_n=e_{n-3},$\\[2mm]

$\mu_{15}^n$&$:$& $e_i\circ e_j=C_{i+j-1}^j,$&$2\leq i+j\leq n-3,$
&$e_1\circ e_n=e_{n-2},$&$e_{n}\circ e_1=e_{n-1},$&$e_n\circ e_n=e_{n-3},$\\[2mm]

$\mu_{16}^n$&$:$& $e_i\circ e_j=C_{i+j-1}^j,$&2$\leq i+j\leq n-3,$
&$e_1\circ e_n=\frac{1}{n-4}e_{n-1},$&$e_{n}\circ e_1=e_{n-3}+e_{n-1},$&$e_n\circ e_n=e_{n-2}.$
\end{longtable}

}


\begin{thebibliography}{99}
\bibitem{ack}
Abdelwahab H.,  Calder\'on A.J., Kaygorodov I.,
The algebraic and geometric classification of nilpotent binary Lie algebras, 
International Journal of Algebra and Computation,     29 (2019), 6, 1113--1129.

\bibitem{abash}
Adashev J., Camacho L., Gomez-Vidal S., Karimjanov I., 
Naturally graded Zinbiel algebras with nilindex $n-3,$ 
Linear Algebra and its Applications, 443 (2014), 86--104.

\bibitem{omirov}
Adashev J.,  Camacho L.,  Omirov B., 
Central extensions of null-filiform and naturally graded filiform non-Lie Leibniz algebras, 
Journal of Algebra, 479 (2017), 461--486. 

\bibitem{adashev}
Adashev J., Khudoyberdiyev A. Kh., Omirov B. A., 
Classifications of some classes of Zinbiel algebras,
Journal of Generalized Lie Theory and Applications, 4 (2010), 10 pages.

\bibitem{adashev19} 
Adashev J., Ladra M., Omirov B., 
The classification of naturally graded Zinbiel algebras with characteristic sequence equal to $(n-p,p)$,  
Ukrainian Mathematical Journal, 71 (2019),  7, 867--883.

\bibitem{bkk}
Bauerle G.G.A., de Kerf E.A.,  ten Kroode A.P.E., 
Lie Algebras. Part 2. Finite and Infinite Dimensional Lie Algebras and Applications in Physics, edited and with a preface by E.M. de Jager, Studies in Mathematical Physics, vol. 7, North-Holland Publishing Co., Amsterdam, ISBN 0-444-82836-2, 1997, x+554 pp.


\bibitem{brem}
Bremner M., 
On Tortkara triple systems, 
Communications in Algebra, 46 (2018),  6, 2396--2404. 

\bibitem{cfk182}
Calder\'on Mart\'{i}n A., Fern\'andez Ouaridi A., Kaygorodov I., 
The classification of $n$-dimensional anticommutative algebras with $(n-3)$-dimensional annihilator, 
Communications in Algebra, 47 (2019), 1, 173--181.

\bibitem{cfk19}
Calderón Martín A.,  Fern\'andez Ouaridi A., Kaygorodov I.,
    The classification of $2$-dimensional rigid algebras,
    Linear and Multilinear Algebra, 68 (2020), 4,  828--844. 
    
    
 

\bibitem{cam13}
Camacho L.,  Ca\~nete E., G\'omez-Vidal S., Omirov B., 
$p$-filiform Zinbiel algebras, 
Linear Algebra and its Applications, 438 (2013), 7, 2958--2972. 

\bibitem{ckkk19}
Camacho L.,  Karimjanov I.,  Kaygorodov I., Khudoyberdiyev A.,
    One-generated nilpotent Novikov algebras, 
     Linear and Multilinear Algebra, 2020, DOI: 10.1080/03081087.2020.1725411.

\bibitem{lisa}
Camacho L., Kaygorodov I.,  Lopatkin V., Salim M., 
    The variety of dual Mock-Lie algebras,  
    Communications in Mathematics, 2020, to appear, arXiv:1910.01484.

\bibitem{degr3}
Cical\`{o} S., De Graaf W.,   Schneider C., 
Six-dimensional nilpotent Lie algebras,
Linear Algebra and its Applications, 436 (2012), 1, 163--189. 


\bibitem{usefi1}
Darijani I., Usefi H., 
The classification of 5-dimensional $p$-nilpotent restricted Lie algebras over perfect fields, I.,
Journal of Algebra, 464 (2016), 97--140. 


\bibitem{degr1}
De Graaf W., 
Classification of nilpotent associative algebras of small dimension, 
International Journal of Algebra and Computation, 28 (2018),  1, 133--161.

\bibitem{degr2}
De Graaf W., 
Classification of 6-dimensional nilpotent Lie algebras over fields of characteristic not 2, 
Journal of Algebra, 309  (2007), 2, 640--653.


\bibitem{dok}
Dokas I., 
Zinbiel algebras and commutative algebras with divided powers,
Glasgow Mathematical Journal, 52 (2010), 2, 303--313.

\bibitem{tortnew1}
 Diehl J.,  Ebrahimi-Fard K.,  Tapia N.,
Time warping invariants of multidimensional time series, 
Acta Applicandae Mathematicae, 2020,  DOI: 10.1007/s10440-020-00333-x

\bibitem{tortnew2}
 Diehl J.,  Lyons T.,  Preis R.,  Reizenstein J.,
Areas of areas generate the shuffle algebra,  arXiv:2002.02338.



\bibitem{dzhuma}
Dzhumadildaev A., 
Zinbiel algebras under $q$-commutators, 
Journal of Mathematical Sciences (New York), 144 (2007), 2, 3909--3925.

\bibitem{dzhuma5}
Dzhumadildaev A., Tulenbaev K., 
Nilpotency of Zinbiel algebras, 
Journal of Dynamical and Control Systems, 11 (2005), 2, 195--213.


\bibitem{dzhuma19}
Dzhumadildaev A., Ismailov N., Mashurov F.,
On the speciality of Tortkara algebras,
Joournal of Algebra, 540 (2019), 1--19.

\bibitem{fkkv19}
Fern\'andez Ouaridi A.,  Kaygorodov I.,  Khrypchenko M., Volkov Yu., 
Degenerations of nilpotent algebras,
arXiv:1905.05361.



\bibitem{gkks}
Gorshkov I., Kaygorodov I., Kytmanov A., Salim M.,
The variety of nilpotent Tortkara algebras, 
Journal of Siberian Federal University. Mathematics \& Physics, 12 (2019), 2, 173--184.



\bibitem{gkp}
Gorshkov I., Kaygorodov I., Khrypchenko M.,
The geometric classification of nilpotent Tortkara algebras,
Communications in Algebra, 48 (2020), 1, 204--209.


\bibitem{gkk}
Gorshkov I., Kaygorodov I., Khrypchenko M.,
The algebraic classification of nilpotent Tortkara algebras,
Communications in Algebra, 48 (2020), 8, 3608--3623


\bibitem{ha16}
Hegazi A., Abdelwahab H., 
Classification of five-dimensional nilpotent Jordan algebras,
Linear Algebra and its Applications, 494 (2016), 165--218.


\bibitem{ha16nm}
Hegazi A., Abdelwahab H., 
Is it possible to find for any $n,m \in \mathbb N$ a Jordan algebra of nilpotency type $(n,1,m)$?, 
Beitr\"{a}ge zur Algebra und Geometrie, 57 (2016), 4, 859--880.

\bibitem{ha17}
Hegazi A., Abdelwahab H., 
The classification of $n$-dimensional non-associative Jordan algebras with $(n-3)$-dimensional annihilator,
Communications in Algebra, 46 (2018), 2, 629--643.

\bibitem{hac16}
Hegazi A., Abdelwahab H., Calder\'{o}n Mart\'{i}n A., 
The classification of $n$-dimensional non-Lie Malcev algebras with $(n-4)$-dimensional annihilator, 
Linear Algebra and its Applications, 505 (2016), 32--56. 

\bibitem{hac18}
Hegazi A., Abdelwahab H., Calder\'{o}n Mart\'{i}n A., 
Classification of nilpotent Malcev algebras of small dimensions over arbitrary fields of characteristic not 2,
Algebras and Representation Theory, 21 (2018), 1, 19--45. 

\bibitem{ikm19}
Ismailov N., Kaygorodov I.,  Mashurov F.,
    The algebraic and geometric classification of nilpotent assosymmetric algebras, 
    Algebras and Representation Theory, 2020, DOI: 10.1007/s10468-019-09935-y.



\bibitem{jkk19}
Jumaniyozov D., Kaygorodov I., Khudoyberdiyev A., 
    The algebraic and geometric classification of nilpotent noncommutative Jordan algebras,
    Journal of Algebra and its Applications, 2020, DOI: 10.1142/S0219498821502029


\bibitem{kkk18}
Karimjanov I., Kaygorodov I.,  Khudoyberdiyev A.,
The algebraic and geometric classification of nilpotent Novikov algebras, 
Journal of Geometry and Physics,   143 (2019), 11--21.


\bibitem{kkl18}
Karimjanov I., Kaygorodov I.,	Ladra M.,
Central extensions of filiform associative algebras, 
Linear and Multilinear Algebra, 2019, DOI: 10.1080/03081087.2019.1620674.


\bibitem{kl20}
Karimjanov I.,  Ladra M., 
    Some classes of nilpotent associative algebras, 
    Mediterranean Journal of Mathematics, 17 (2020), 2, Pape  70, 21 pp.


\bibitem{kkl20}
Kaygorodov I., Khrypchenko M., Lopes S.,
    The algebraic and geometric classification of nilpotent anticommutative algebras,
    Journal of Pure and Applied Algebra, 224 (2020), 8, 106337.

\bibitem{kks19}
Kaygorodov I., Khudoyberdiyev  A., Sattarov A.,  
    One-generated nilpotent terminal  algebras, Communications in Algebra, 48 (2020),  10, 4355--4390.
    
\bibitem{klp20}
Kaygorodov I., Lopes S., P\'{a}ez-Guill\'{a}n P.,  
    Non-associative central extensions of null-filiform associative algebras, 
    Journal of Algebra, 560 (2020),  1190--1210.


\bibitem{kpv19}
Kaygorodov I., P\'{a}ez-Guill\'{a}n  P., Voronin V.,  
The algebraic and geometric classification of nilpotent bicommutative algebras,
Algebras and Representation Theory, 2020, DOI: 10.1007/s10468-019-09944-x.



\bibitem{kppv}
Kaygorodov I.,  Popov Yu., Pozhidaev A., Volkov Yu., 
Degenerations of Zinbiel and nilpotent Leibniz algebras, 
Linear and Multilinear Algebra,   66 (2018), 4, 704--716.





\bibitem{kv16}
Kaygorodov  I.,  Volkov  Yu.,  
    The  variety  of $2$-dimensional  algebras  over  an  algebraically  closed  field, Canadian Journal of Mathematics, 71 (2019), 4, 819--842.


\bibitem{loday}
Loday J.-L., 
Cup-product for Leibniz cohomology and dual Leibniz algebras, 
Mathematica Scandinavica, 77 (1995), 2, 189--196.


\bibitem{mukh}
Mukherjee G.,  Saha R., 
Cup-product for equivariant Leibniz cohomology and Zinbiel algebras, 
Algebra Colloquium, 26 (2019), 2, 271--284.


\bibitem{anau}
Naurazbekova A., 
On the structure of free dual Leibniz algebras, 
Eurasian Mathematical Journal, 10 (2019), 3,  40--47.

\bibitem{ualbay}
Naurazbekova A., Umirbaev U., 
Identities of dual Leibniz algebras, 
TWMS Journal of Pure and Applied Mathematics, 1 (2010),  1, 86--91.



\bibitem{ss78}
Skjelbred T., Sund T., 
Sur la classification des algebres de Lie nilpotentes, 
C. R. Acad. Sci. Paris Ser. A-B, 286 (1978), 5, A241--A242.

\bibitem{is11}
Rakhimov I.,  Hassan M., 
On one-dimensional Leibniz central extensions of a filiform Lie algebra,
Bulletin of the Australian Mathematical Society, 84 (2011),  2, 205--224. 

\bibitem{yau}
Yau D., 
Deformation of dual Leibniz algebra morphisms, 
Communications in Algebra, 35 (2007),  4, 1369--1378. 


\bibitem{zusmanovich}
Zusmanovich P., 
Central extensions of current algebras, 
Transactions of the American Mathematical Society,  334 (1992),  1, 143--152. 

\end{thebibliography}
\end{document}